\documentclass[runningheads,11pt]{llncs}
\usepackage{makeidx}

\usepackage{amsmath, amsfonts, amssymb}
\usepackage[non-compressed-cites, non-sorted-cites]{amsrefs}
\usepackage{hyperref}
\usepackage{enumitem}

\title{Reverse Mathematics of Matroids}
\titlerunning{Reverse Mathematics of Matroids}
\toctitle{Reverse Mathematics of Matroids}

\author{
Jeffry L. Hirst\inst{1}
\and
Carl Mummert\inst{2}
}
\institute{
Appalachian State University, Boone, NC 28608, USA,\\
\email{hirstjl@appstate.edu},\\
WWW home page:  \texttt{http://mathsci.appstate.edu/\homedir{}jlh/ }
\and
Marshall University, Huntington, WV 25755, USA,\\
\email{mummertc@marshall.edu},\\
WWW home page:  \texttt{http://m6c.org/w/}
}

\authorrunning{Jeffry L. Hirst and Carl Mummert}
\tocauthor{Jeffry L. Hirst and Carl Mummert}

\newenvironment{enumlist}{\begin{enumerate}[label=\upshape{(\arabic*)}]}{\end{enumerate}}

\spnewtheorem{coro}[theorem]{Corollary}{\bfseries}{\itshape}  %Corollaries, lemmas and definitions are numbered with theorems
\spnewtheorem{lem}[theorem]{Lemma}{\bfseries}{\itshape}
\spnewtheorem{defn}[theorem]{Definition}{\bfseries}{}            %Definitions in roman type, chosen because of length

\newcommand{\nat}{\mathbb N}  %this defines \nat as an abbreviation for \mathbb N
\newcommand{\rca}{{\sf RCA}_0}
\newcommand{\aca}{{\sf ACA}_0}

\newcommand{\mfin}{M^{<\nat}}
\newcommand{\zej}{{\bf 0}^\prime}
\newcommand{\clo}{{\text{cl}}}

\begin{document}

\maketitle

\begin{abstract}
Matroids generalize the familiar notion of linear dependence  from linear algebra.
Following a brief discussion of founding work in computability and matroids,
we use the techniques of reverse mathematics to determine the logical
strength of some basis theorems for matroids and enumerated matroids.
Next, using Weihrauch reducibility, we relate the basis results to 
combinatorial choice principles and statements about vector spaces.
Finally, we formalize some of the Weihrauch reductions to extract related
reverse mathematics results.  In particular, we show that the existence of
bases for vector spaces of bounded dimension is equivalent to the induction scheme
for $\Sigma^0_2$ formulas.\\[1em]
\noindent \textbf{Keywords:} Reverse mathematics, matroid, induction, graph, connected component\\[1em]
\textbf{MSC Subject Class (2000):} 03B30; 03F35; 05B35
 \end{abstract}

 The study of computable and computably enumerable matroids links the work in 
 this paper to the theme of this volume.  The following incomplete survey establishes a
 framework for this connection and provides a few pointers into the substantial literature on
 computability and matroids.
 
 In a seminal paper on  
 computable
 and c.e.~vector spaces,  Metakides and Nerode~\cite{metakides} 
  defined a vector space $V_\infty$, the $\aleph_0$-dimensional vector space
 over a countable computable field $F$ consisting of $\omega$-sequences of
 elements of $F$ with finite support, with point-wise operations.   The lattice of c.e.~subspaces 
 of $V_\infty$ is denoted $\mathcal L (V_\infty )$.  A vector space $V$
 over a computable field $F$ is \emph{c.e.~presented} if it has an effective enumeration of
 the vectors, partial recursive addition and scalar multiplication operations, and a
 c.e.~congruence relation $\equiv$ such that the quotient $V/\mathbin{\equiv}$ is a vector space.
 Metakides and Nerode proved that a vector space is c.e.~presented if and only if
 it is computably isomorphic to $V_\infty/ W$ for some $W\in \mathcal L (V_\infty)$.
 
  Many proofs of results for $\mathcal L (V_\infty )$ rely on the structure of $V_\infty$,
 hampering their adaptation to $\mathcal L (F_\infty)$, the lattice of
 c.e.~algebraically closed subfields of a sufficiently computable algebraically closed
 field $F_\infty$ with countably infinite transcendence degree.
 Matroids restrict interest to dependence properties common to both
 vector spaces and algebraic extensions, so proofs based on matroids can
 often be adapted to both vector space and field settings.
 
 In computability theoretic papers, matroids are often described in terms of \emph{Steinitz systems}.
 These are also called Steinitz  \emph{closure} systems~\cite{mnjoa} or
 Steinitz \emph{exchange} systems~\cite{nrpatras}.  Downey~\cite{dnws} defines
 a Steinitz system as a set $U$ and a closure operator $\clo$ mapping subsets
 of $U$ to subsets of $U$ such that if $A,B \subset U$,
 \begin{enumlist}
 \item $A \subset \clo (A)$,
 \item $A\subset B$ implies $\clo(A) \subset \clo (B)$,
 \item $\clo(\clo (A))= \clo (A)$,
 \item $x\in \clo (A) $ implies that, for some finite $A^\prime \subset A$, $x \in \clo (A^\prime )$, and
 \item (exchange) $x \in \clo (A \cup \{y\}) - \clo (A)$ implies $y \in \clo ( A \cup \{ x\} )$.
 \end{enumlist}
 As an intuitive example, we can think of $U$ as a vector space and $\clo(A)$ as 
  the linear span of the vectors in the set $A$.
 The Steinitz system $(U, \clo )$ has \emph{computable dependence} if $U$ is computable
 and there is a uniformly effective procedure that, when applied to $a, b_1 , \dots b_n \in U$,
 computes whether $a \in \clo ( \{ b_1, \dots b_n \} )$.
 
 A central goal in computable matroid research to discover algebraic properties of
 matroids with significant computability theoretic consequences.  For example, the
 Steinitz system
 $(U,\clo )$ has the \emph{closure intersection property} if whenever
 \begin{list}{$\bullet$}{}
 \item $D$ is closed, that is, $\clo (D) = D$,
 \item $A$ is independent over $D$, that is, for every $a \in A$, $a \notin \clo (D\cup A\setminus \{a\})$,
 \item $B$ is independent over $D$, and 
 \item $\clo (A \cup D) \cap \clo (B\cup D) = \clo (D)$,
 \end{list}
 then $A \cup B$ is independent over $D$.  The system is \emph{semiregular}
 (called \emph{Downey's semiregularity} by Nerode and Remmel~\cite{nrpatras}) if no finite dimensional
 closed set is the union of two closed proper subsets.  In his thesis~\cite{dthesis}
 (abstracted in~\cite{dbams}), Downey established that if $(U, \clo )$ is semiregular
 and has the closure intersection property then the theory of $\cal L (U)$ is
 undecidable.

\section{Reverse Mathematics}
 
In his development of the theory of matroids, Whitney~\cite{whitney}*{Section 6}
formulates matroids in terms of a ground set of elements and a specification
of every set as being either dependent or independent.  We define an
\emph{enumerated} matroid (\emph{e-matroid}) to consist of a set and an enumeration
of its finite dependent sets.
 
 \begin{defn}
 An \emph{e-matroid} is a pair $(M,e)$ consisting of a set $M$ and a function
 $e\colon \nat \to \mfin$ satisfying:
 \begin{enumlist}
 \item \label{em1} The empty set is independent.
 \[
 (\forall n)[ e(n) \neq \emptyset ]
 \]
 
 \item \label{em2} Finite supersets of dependent sets are dependent.
 \[
(\forall n )(\forall Y\in \mfin)[e(n)\subseteq Y \to \exists m (e(m)=Y)]
\]
 \item \label{em3} If $X$ is an independent set that is smaller than an independent set $Y$,
 then $Y$ contains an element that is independent of $X$.
 \[
 \begin{split}
(\forall X,Y \in \mfin)(\,\text{if } &  |X|<|Y|  \text{ and  } (\forall n)[e(n) \neq X \land e(n) \neq Y] \\
& 
\text{ then } (\exists y\in Y)( \forall n) [e(n) \neq X \cup \{y\}])
 \end{split}
 \]
  \end{enumlist}
 \end{defn}

Although dependence in this setting is not directly related to linear combinations,
it is still possible to formulate concepts of span and bases.

\begin{defn}
A subset $B$ of an e-matroid $(M,e)$ \emph{spans} the e-matroid if adjoining any additional element
to $B$ produces a dependent set, that is, 
\[
(\forall x \in M )[x \notin B \to (\exists n)(e(n)=B\cup\{x\})].
\]
A subset $B \subseteq M$ is a \emph{basis} for the e-matroid  if
$B$ is independent (that is, $(\forall n)[e(n) \not\subseteq B]$)
and $B$ spans~$M$.
 \end{defn}

We can now state our first basis theorem.  The analogous result showing the equivalence of $\aca$
and the existence of bases for vector spaces is included in Theorem 4.3 of
Friedman, Simpson, and Smith~\cite{fss}.

\begin{theorem}\label{b1}
$(\rca )$ The following are equivalent:
\begin{enumlist}
\item \label {b1a} $\aca$.
\item \label {b1b} Every e-matroid has a basis.
\end{enumlist}
\end{theorem}

\begin{proof}
To show that \ref{b1a} implies \ref{b1b}, fix an e-matroid $(M,e)$. %LNCS change
Let $m_0, m_1 , \dots$ be a non-repeating enumeration of $M$.  Consider the
function $g\colon \nat \to\mfin$ defined by $g(0)= \emptyset$ and for $i>0$,
\[
g(i)=
\begin{cases}
g(i-1) &\text{if~}  (\exists n )[e(n) = g(i-1) \cup \{ m_{i-1} \}],\\
g(i-1)\cup\{m_{i-1} \}&\text{otherwise.}
\end{cases}
\]
by arithmetical comprehension, the union of the range of $g$ exists; call this union $B$.
Straightforward arguments verify that $B$ is a basis for $M$.

To prove the converse, by Lemma III.1.3 of Simpson~\cite{simpson}, it suffices to use~\ref{b1b}
to prove the existence of the range of an arbitrary injection from $\nat$ to~$\nat$. 
Suppose $f\colon \nat \to \nat$ is an injection.
Let $M=\{ (i,\varepsilon ) : i \in \nat \land \varepsilon < 2 \}$ be the ground set for an e-matroid.
Let $M_0 , M_1 , \dots$ be an enumeration of $\mfin$.   Fix a bijective pairing function mapping
$\nat \times \nat$ onto $\nat$.
Using the notation $(j,k)$ for both the
pair and its integer code, define
$e((j,k)) = \{ (f(j),0) , (f(j),1)\} \cup M_k$. Because $(f(j),0)\in e((j,k))$, item~\ref{em1} of the definition
of an e-matroid holds.  The inclusion of $M_k$ in $e((j,k))$ ensures that supersets of dependent sets
are dependent, satisfying item~\ref{em2} of the definition.
To verify item~\ref{em3}, suppose $X$ and $Y$ are finite independent sets with $|X|<|Y|$.
If there is a $y \in X \cap Y$, then $X\cup \{y \} = X$ so $\forall n (e(n) \neq X \cup \{y\} )$.
Thus we need only consider the case where $X \cap Y = \emptyset$.
We hypothesized that $|Y|>|X|$, so there must be a $y =(z,\varepsilon ) \in Y$ such that for all
$\varepsilon^\prime$, $(z , \varepsilon^\prime ) \notin X$.  
Suppose by way of contradiction that $e(n) = X\cup \{y\}$ for some~$n$.
Then, for some $j$, we have $\{ (f(j),0),(f(j),1)\}\subset X \cup \{y\}$.  By the choice of $y$,
we know $f(j) \neq z$, so $\{(f(j),0) , (f(j),1) \} \subset X$, contradicting $(\forall n)[e(n) \neq X]$.
Thus item~\ref{em3} of the definition holds, and we have shown that $(M,e)$ is an e-matroid.

Finally, we claim that if $B$ is a basis for $M$, then $k$ is in the range of $f$ if and only if 
$(k,0) \notin B$ or $(k,1) \notin B$.  First note that if $k = f(j)$ then, assuming $0$ is the code
for $\emptyset$, we have $e((j,0)) = \{ (k,0),(k,1)\}$.
$B$ is a basis, so $e((j,0)) \not\subset B$, and thus $(k,0) \notin B$ or $(k,1) \notin B$.
Conversely, if for example $(k,0) \notin B$, then
$(\exists n)[ e(n) = B\cup \{(k,0)\}]$.  Because $e(n)$ is dependent and
$B$ is independent, both $(k,1) \in e(n)$ and for all $j \neq k$,
at least one of $(j,0)$ and $(j,1)$ is not in $e(n)$.  By the definition of $e$, $e(n)$ must
contain both $(a,0)$ and $(a,1)$ for some $a$ in the range of $f$, so $k$ is in the range of $f$.
A similar argument holds if $(k,1) \notin B$, completing the proof of our claim.  Because
$k$ is in the range of $f$ if and only if 
$(k,0) \notin B$ or $(k,1) \notin B$,
recursive comprehension suffices to prove the existence of the range of $f$, completing the reversal.
\end{proof}

Our next result shows that if we add a hypothesis bounding the dimension of the matroid, the principle asserting the existence of a basis becomes weaker.  
The result also illustrates the interrelatedness of matroids and
graph theory. We use the concept of rank to establish the dimensional bound.

\begin{defn}
We say the \emph{rank} of an e-matroid $(M,e)$ is \emph{no more than}~$n$ if every subset of $M$ of size $n$ is dependent,
that is, in the range of $e$.
\end{defn}

\begin{theorem}\label{b2}
$(\rca )$ The following are equivalent:
\begin{enumlist}
\item \label{b2a} For every $n$, every e-matroid of rank no more than~$n$ has a basis.
\item \label{b2b} For every $n$, if $G=(V,E)$ is a countable graph and every collection
of $n$ vertices contains at least one path connected pair, then $G$ can be decomposed
into its connected components.
\item \label{b2c} ${\sf I}\Sigma^0_2$, the induction scheme for $\Sigma^0_2$ formulas with set parameters.
\end{enumlist}
\end{theorem}

\begin{proof}
Proofs that~\ref{b2b} implies~\ref{b2c} appear as Theorem 4.5 of Hirst~\cite{hirst1992} and also as
Theorem 3.2 of Gura, Hirst, and Mummert~\cite{gura}.  Here, we will prove that~\ref{b2c} implies
\ref{b2a} and~\ref{b2a} implies~\ref{b2b}.

To see that~\ref{b2c} implies~\ref{b2a}, fix $n$ and let $(M,e)$ be an e-matroid of rank no more than~$n$.
Let $\psi (j)$ formalize the existence of an independent set of size $n-j$.  If we use $X_t$ to denote
the finite subset of $\nat$ encoded by $t$, then $\psi (j)$ can be written as
\(
(\exists t)[|X_t |= n-j \land \forall k (e(k)\neq X_t) ].
\)
Note that $\psi(j)$ is a $\Sigma^0_2$ formula, and the empty set witnesses $\psi (n)$.  By the $\Sigma^0_2$
least element principle (which is easily deduced from the bounded $\Sigma^0_2$ comprehension, and is therefore a consequence of~\ref{b2c} by Exercise II.3.13 of Simpson~\cite{simpson}), there is a least
$j_0$ such that $\psi(j_0 )$.  Let $X_{t_0}$ witness $\psi (j_0 )$.  We claim that $X_{t_0}$ is a basis.
The range of $e$ is closed under supersets, so no subset of $X_{t_0}$ appears in the range of $e$.  By the minimality
of $j_0$, if $x\notin X_{t_0}$, then $X_{t_0} \cup \{ x \}$ is dependent, so for some $n$, $e(n) = X_{t_0} \cup \{ x\}$.
Thus $X_{t_0}$ spans $M$.

To show that~\ref{b2a} implies~\ref{b2b}, let $G(V,E)$ be a graph in which every collection of $n$ vertices contains
at least one path connected pair.  The independent sets of our e-matroid will consist of subsets of $V$ with no
path connected pairs.  If $G$ contains no edges, the identity function on $V$ decomposes $G$ into connected
components.  Suppose $G$ has an edge connecting the vertices $v_0$ and $v_1$.  Let $(V_i )_{i \in \nat}$ be
an enumeration of the finite subsets of $V$ such that every subset appears infinitely often.  Define $e(j)$ by
$e(j)= V_j$ if there is some $t<j$ that encodes a path between two vertices of $V_j$, and $e(j) = \{ v_0 , v_1\}$
otherwise.  It is easy to verify that $(V,e)$ satisfies the first two clauses of the definition of an e-matroid.
For the third clause, suppose $X$ and $Y$ are finite sets of vertices such that no pair in either set is
path connected, and that $|X|<|Y|$.  Suppose by way of contradiction that every vertex in $Y$ is path connected
to some vertex in $X$.  $\rca$ can prove the existence of the function mapping each $y\in Y$ to some $x\in X$
to which it is path connected, and because $|X|<|Y|$, $f$ must map two elements of $Y$ to the same $x$.
These two vertices of $Y$ are path connected, yielding the desired contradiction.  Thus $(V,e)$ is a matroid.
By~\ref{b2a}, $(V,e)$ has a basis, which is a maximal set of disconnected vertices in $G$.  The function which
is the identity on this basis and maps very other vertex of $G$ to the element of the basis to which it is path
connected is a decomposition of $G$ into connected components.  This decomposition is computable from
the basis, so $\rca$ proves~\ref{b2a} implies~\ref{b2b}.
\end{proof}

\section{Why e-Matroids?}

We can define a matroid as a pair $(M,D)$ where $D$ is the set of all finite dependent subsets of $M$.
In this case, $D$ satisfies the set-based analogs of the three items in the definition of e-matroid. 
To express this definition within $\rca$, we represent each finite subset of $M$ via its characteristic index.
 Using
the set-based analog of the definition of basis, we can state and prove the following result.

\begin{theorem}\label{basis1}
$(\rca )$  Every matroid has a basis.
\end{theorem}

\begin{proof}
Let $(M,D)$ be a matroid and let $m_1, m_2, \dots$ be a non-repeating enumeration of $M$.
Define a nested sequence of finite independent sets $\langle I_j \rangle_{j \in \nat}$ as follows.
Let $I_0 = \emptyset$.  For $j>0$, let $I_j = I_{j-1}$ if $I_{j-1}\cup \{m_j\} \in D$, and let
$I_j = I_{j-1} \cup \{m_j\}$ otherwise.  Define the basis $B$ by $m_j\in B$ if and only if
$m_j \in I_j$.  To see that $B$ is independent, suppose $X$ is a finite dependent set.
Let $m_j$ be the element of largest index in $X$.  If $X \setminus \{m_j\} \subset I_{j-1}$,
then $m_j \notin I_j$, so $m_j \notin B$ and $X \not\subset B$.  If $X \setminus \{m_j\} \not\subset I_{j-1}$
then $X \not\subset I_j$, so $X \not\subset B$.  Summarizing, $B$ has no finite dependent subsets,
so $B$ is independent.  To see that $B$ spans, fix $m_j \in M$.  Either $m_j \in B$, or both
$B\supset I_{j-1} \notin D$ and $I_{j-1} \cup \{m_j\} \in D$.  In either case, $m_j$ is in the span of $B$.
\end{proof}

The preceding result can be viewed as a reverse mathematical reframing of the statement:
\emph{Every computably presented matroid has a computable basis.}  This principle was
stated of Crossley and Remmel~\cite{crossley}*{\textsection{}5, Lemma
1}, who describe it as common knowledge
and implicit in the work of Metakides and Nerode~\cite{metakides}.  
The representations of the matroid by a computable dependence relationship or by a
dependence algorithm for a Steinitz system with computable dependence are equivalent.
The next theorem is a
reverse mathematics analog of the fact that not every c.e.~presented matroid is computably
isomorphic to a computably presented matroid.

\begin{theorem}\label{iso}
$(\rca )$  The following are equivalent:
\begin{enumlist}
\item  \label{isoa}$\aca$.
\item \label{isob}Every e-matroid is isomorphic to a matroid.  
That is, if $(M,e)$ is an e-matroid, then there
is a matroid $(N, D)$ and a bijection $h\colon M\to N$ such that for all finite
sets $X \subset M$, there is an $n$ such that $e(n)=X$ if and only if
$\{ h(x) :  x\in X\} \in D$.
\end{enumlist}
\end{theorem}

\begin{proof}
To see that~\ref{isoa} implies~\ref{isob}, suppose $(M,e)$ is an e-matroid.
The range of $e$ is arithmetically definable using $e$ as a parameter,
so $\aca$ proves the existence of the range as a set $D$.  Then $(M,D)$ is a matroid
and the identity is the desired isomorphism.

To prove the converse, we capitalize on the construction from the proof of the reversal
of Theorem~\ref{b1}.  As in that proof, fix an injection $f$ and construct the associated e-matroid
$(M,e)$.  Apply~\ref{isob} above to find a matroid $(N,D)$ and an isomorphism $h\colon M \to N$.
By the construction of $(M,e)$, for each $k\in \nat$, $k$ is in the range of $f$ if and only if
$\{(k,0),(k,1)\}$ is in the range of~$e$, which holds if and only if $\{h((k,0)),h((k,1))\}\in D$.
Thus, the range of $f$ is computable from $D$ and $h$, completing the proof of the reversal.
\end{proof}

In terms of Turing degrees, the previous theorem  only shows
that each c.e.~presented matroid is computable from~$\zej$.  
The next corollary shows that, if a c.e~presented matroid is isomorphic
to a computable matroid, the isomorphism may necessarily be noncomputable. 

\begin{coro}
There is a c.e.~presented matroid $M$, which is isomorphic to a computable matroid, 
such that if $\varphi$ is any isomorphism between $M$ and a computable matroid 
then $\zej$ is Turing computable from $\varphi$.
\end{coro}

\begin{proof}
Let $f$ be any computable injection with a range that computes $\zej$.  Use the construction
of $(M,e)$ from the proof of the reversal of Theorem~\ref{b1}.  This is the desired c.e.~presented
matroid.  The proof of Theorem~\ref{iso} shows that any isomorphism between $(M,e)$ and a
computable matroid computes the range of $f$ and consequently computes $\zej$.  Since the
range of $f$ is both infinite and co-infinite, $(M,e)$ is isomorphic to the computable matroid with
ground set $\nat$ and $D$ consisting of all finite supersets of sets of the form $\{3k , 3k+1\}$
where $k\in \nat$.
\end{proof}

\section{Weihrauch Reducibility}

In Theorem~\ref{basis1}, we used Reverse Mathematics to study the problem 
of finding a basis for an e-matroid. In this section, we study the same problem using
 Weihrauch reducibility. For additional information on Weihrauch reducibility, see 
Brattka and Gherardi~\cite{bg2011} and Dorais, Dzhafarov, Hirst, Mileti, and Shafer~\cite{ddhms2016}. The following simplified definition of Weihrauch problems will be sufficient for our purposes.

\newcommand{\equivsW}{\equiv_{\textup{sW}}}
\newcommand{\leqsW}{\leq_{\textup{sW}}}

\newcommand{\equivW}{\equiv_{\textup{W}}}
\newcommand{\leqW}{\leq_{\textup{W}}}

\begin{defn} 
A \emph{Weihrauch problem} is a subset of $\nat^\nat \times \nat^\nat$, $\nat^\nat \times \nat$,
$\nat \times \nat^\nat$, or $\nat \times \nat$. For a Weihrauch problem $P$, the ``problem'' is:
given an ``instance'' $I \in \operatorname{dom}(P)$, produce a ``solution'' $S$ with $(I,S) \in P$. 

A Weihrauch problem $P$ is \emph{Weihrauch reducible} to a Weihrauch problem $Q$, written
$P \leqW Q$, if there
are computable functions or functionals $\Phi, \Psi$ such that, for all $S \in \operatorname{dom}(P)$,
$\Phi(S) \in \operatorname{dom}(Q)$, and for all $R$ such that $(\Phi(S), R) \in Q$,
we have $(S, \Psi(R,S)) \in P$.  If this can be done with a functional $\Psi$ that does not
depend on $S$, we say that $P$ is \emph{strongly Weihrauch reducible} to $Q$, written $P \leqsW Q$.
The relations $\leqW$ and $\leqsW$ are reflexive and transitive, and thus they induce equivalence relations, which are denoted $\equivW$ and $\equivsW$, respectively. 

The \emph{parallelization} of a Weihrauch problem $P$ is the problem 
\[
\widehat P
= \{ (f, g) : (f(n) , g(n)) \in P \text{ for all } n \in \nat \}
\]
whose instances are sequences of instances of $P$ and whose solutions are
sequences of solutions corresponding to those instances. 

\end{defn}

\begin{defn}
We define the following Weihrauch principles. The first two are well known in the literature~\cite{bbp2012}.
\begin{itemize}
\item $\mathsf{C}_{\nat}$: closed choice for subsets of $\nat$.
\[
\mathsf{C}_\nat = \{ (f, n) : f \in \nat^\nat, n \not \in \operatorname{range}(f) \}
\]

\item $\widehat{\mathsf{C}}_{\nat}$: the parallelization of $\mathsf{C}_\nat$.
\[
\widehat{\mathsf{C}}_\nat = \{ (f, g) :  ((f)_n,g(n)) \in \mathsf{C}_\nat \text{ for all $n \in \nat$ } \}
\]

\item $\mathsf{GAC}$: the graph antichain problem. For a countable graph $G$, 
an antichain is a set of vertices no two of which are connected by a path in~$G$.
Letting $\operatorname{Max}(G)$ be the set of maximal antichains of $G$, we have
\[
\mathsf{GAC} = \{ (G, A) : G \text{ is a countable graph}, A \in \operatorname{Max}(G)\}
\]

\item $\mathsf{EMB}$: the e-matroid basis problem.
\[
\mathsf{EMB} = \{ (M,B): M \text{ is a countable e-matroid}, B \text{ is a basis for $M$}\}
\]

\item $\mathsf{VSB}$: the vector space basis problem. 
\[
\textsf{VSB} = \{ (V,B) : V \text{ is a countable vector space and } B \text{ is a basis for $V$}\} 
\]

\end{itemize}
For each $n > 1$ in $\nat$, we define the following  restricted principles:
\begin{itemize}
\item $\mathsf{GAC}_n$: the restriction to $\mathsf{GAC}$ to graphs with $n$ connected components. 
\item $\mathsf{EMB}_n$: the restriction of $\mathsf{EMB}$ to e-matroids with a basis of size~$n$. 
\item $\mathsf{VSB}_n$: the restriction of $\mathsf{VSB}$ to vectors spaces with dimension~$n$.
\end{itemize}

\end{defn}

In previous work~\cite{gura}, we considered another well known
Weihrauch problem, $\mathsf{LPO}$.
\[
\mathsf{LPO} = \{ (f,n) : f \in \mathbb{N}^\mathbb{N} \text{ and } 
f(n) = 0 \leftrightarrow (\exists m)[f(m) = 0]\}
\] 
  The following lemma shows that the parallelization of $\mathsf{LPO}$ is strict Weihrauch equivalent
to the parallelization of $\mathsf{C}_\mathbb{N}$. This equivalence is implicit in work of Brattka and Gherardi~\cites{bg2011, bg2011b}, but the reductions obtained by combining their results are very indirect. The next lemma provides a pair
of direct reductions. 

\begin{lem}
$\widehat{\mathsf{C}}_\nat$ is strongly Weihrauch equivalent to $\widehat{\mathsf{LPO}}$.
\end{lem}
\begin{proof}
First, suppose we are given an instance $f$ of $\mathsf{C}_\nat$. The function $f$ enumerates
the complement of some nonempty set. We form a sequence $(p_n)$ of instances of $\mathsf{LPO}$
such that $p_n$ has $0$ in its range if and only if $n$ is in the range of~$f$. Then, given solutions to the instance $(p_n)_{n \in \nat}$ of $\widehat{\mathsf{LPO}}$, we can search effectively for the least $n$ such that $p_n$ does not have $0$ in its range, which will be the least $n$ not in the range of~$f$. Thus, by effective dovetailing, $\widehat{\mathsf{C}}_\mathbb{N}$ is strict Weihrauch reducible to~$\widehat{\mathsf{LPO}}$.  

 For the converse, we first reduce $\mathsf{LPO}$ to $\mathsf{C}_\nat$, as follows. Given an instance $p$ of $\mathsf{LPO}$, we enumerate in stages the complement of a nonempty set $A = A(p)$. If $p(0) > 0$, we enumerate $1$ into the complement of A. Then if $p(1) > 0$ we enumerate 2 into the complement of A.  We continue in this way. If we ever find that $p(n) = 0$ for some $n$, we enumerate $0$ into the complement of $A$, after which we do not enumerate anything else into the complement, so we will have $A = \{n+1, n+2, \ldots\}$. On the other hand, if $0$ is not in the range of $p$, then we continue enumerating elements into the complement of $A$, so that we will obtain $A = \{ 0 \}$. Hence, if we view $A$ as an instance of $\mathsf{C}_\nat$, we can determine whether $(\exists m)[p(m) =0]$  by looking at the value of any solution.  Thus $\mathsf{LPO}$ is strict Weihrauch reducible to $\mathsf{C}_\nat$, and so the parallelization of $\mathsf{LPO}$ is strict Weihrauch reducible to the parallelization of~$\mathsf{C}_\nat$. 
\end{proof}

\begin{theorem}\label{w1} The following strong Weihrauch equivalences hold:
\[
\mathsf{GAC} \equivsW \mathsf{EMB} \equivsW \mathsf{VSB} \equivsW \widehat{\mathsf{C}}_\nat.
\]
\end{theorem}
\begin{proof}
Gura, Hirst, and Mummert~\cite{gura} proved that $\mathsf{GAC} \equivsW \widehat{\mathsf{C}}_\nat$.
Therefore, it is sufficient to establish the following four reductions:
\[
\mathsf{GAC} \leqsW \mathsf{EMB} \leqsW \widehat{\mathsf{C}}_\nat, 
\qquad
 \widehat{\mathsf{C}}_\nat \leqsW   \mathsf{VSB} \leqsW  \mathsf{EMB}.
\]
Three of these reductions are straightforward. First, to show that  $\mathsf{VSB} \leqsW \mathsf{EMB}$, 
modify the construction used to prove~\ref{b2a} implies~\ref{b2b} in Theorem~\ref{b2}.  Given a vector
space with vectors $V$, let $( V_i )_{i \in \nat}$ be an enumeration of all the finite subsets of $V$ in which
each subset appears infinitely often.  Define $e\colon \nat \to V^{<\nat}$ by setting $e(j) = V_j = \{ v_0 , \dots , v_k\}$ 
if there
is a sequence of field elements $\{ a_0 , \dots , a_k\}$ with canonical code less than $j$ such that
$\sum_{i\le k } a_i v_i = 0$, and set $e_j = \{ \vec 0 \}$ otherwise.  Because $e$ enumerates the finite
dependent subsets of $V$, it is easy to verify that
$(V,e)$ is a matroid and any basis for the matroid is a basis for the vector space.

Second, to show that $\mathsf{GAC} \leqsW \mathsf{EMB}$,
let $G= (V,E)$ be a graph. We wish to ensure that $G$ has at least one edge.
To this end, suppose $v_1 \in V$ and add a new vertex $v_0$ to $V$ and a new edge $(v_0, v_1)$ to $E$, 
yielding a graph $G' = (V', E')$. 
  Construct a matroid
$(V', e)$ as in the proof that~\ref{b2a} implies~\ref{b2b} in Theorem~\ref{b2}.  (Note that in that argument,
the bound on the number of components is used only to bound the rank of the matroid.)  As in that proof,
any basis for $(V',e)$ is a maximal set of disconnected vertices of $G'$.  If
$v_0$ is in the basis, it can be replaced by $v_1$ to form a new basis 
which is a maximal set of disconnected vertices of~$G$.

Third, to show that $\mathsf{EMB} \leqsW \widehat{\mathsf{C}}_\nat$, we 
let $M$ be a countable e-matroid with dependency function $e_M$. For each
finite set $F \subseteq M$, let $S_n$ be the set of all $n > 0$ such that, if 
$e_M(t) = F$ then there is a $t' < n$ with $e_M(t') = F$.  Let $f_n$ be a function
with $\operatorname{range}(f) = \nat \setminus S_n$. Then $(f_n)_{n \in \nat}$ is an instance of
$\widehat{\mathsf{C}}_\nat$. We can compute a basis for $M$, uniformly, from any solution to 
this instance. 

It remains to show that $\widehat{\mathsf{C}}_\nat \leqsW \mathsf{VSB}$. 
We adapt the construction presented by Simpson~\cite{simpson}*{Theorem~III.4.3} showing that the principle
``every countable vector space over $\mathbb{Q}$ has a basis'' is equivalent to $\mathsf{ACA}_0$ in the sense of Reverse Mathematics.  The proof presented by Simpson shows, more specifically, that given an injective function $f \colon \nat \to \nat$ we may uniformly compute a $\mathbb{Q}$-vector space $V_f$ such that the range of $f$ is uniformly computable from any basis of~$V_f$. This shows, in particular, that $\mathsf{C}_\nat \leqsW  \mathsf{VSB}$. 

To complete the proof, it is sufficient for us to verify that $\widehat{\mathsf{VSB}} \leqsW \mathsf{VSB}$,
because then we have $\widehat{\mathsf{C}}_\nat \leqsW \widehat{\mathsf{VSB}} \leqsW \mathsf{VSB}$.
The proof uses an effective direct sum construction. 
Given a sequence $(V_n)_{n \in \nat}$ of countable vector spaces, we may assume without loss of generality that their underlying sets of vectors are pairwise disjoint. We may then form a countable vector space $V$ whose elements are finite formal $\mathbb{Q}$-linear combinations of the form
\[
a_1 u_1 + \cdots + a_m u_m 
\]
where $a_i \in \mathbb{Q}$ and $u_i \in V_i$ for $i \leq m$.  The scalar multiplication on $V$ is the obvious one, and the vector addition is so that 
\[
\left ( \sum_{i \leq m} a_i u_i \right ) + \left ( \sum_{i \leq n} b_i v_i \right ) 
= \sum_{i \leq \max{m,n}} (a_i u_i + b_i v_i) 
\]
where each addition $a_iu_i + b_iv_i$ is carried out in $V_i$, and terms that did not appear in the left
are treated vacuously as zero vectors.  Then $V$ is a countable vector space that is uniformly computable
from the sequence $(V_n)_{n\in \nat}$.  Moreover, if $B$ is a basis of $V$ then $B \cap V_i$ is a basis of $V_i$ for each~$i \in \nat$. To see this, note that on one hand $B \cap V_i$ must span $V_i$ for each~$i$, and on the
other hand any dependency of the set $B \cap V_i$ within $V_i$ would induce a dependency of $B$ within $V$. 
\end{proof}

We next consider the restricted versions of two principles from Theorem~\ref{w1}.

\begin{theorem}\label{w2} For $n \geq 2$,  the following equivalences hold:
\[
\mathsf{GAC}_n \equivsW \mathsf{EMB}_n \equivsW \mathsf{C}_\nat.
\]
\end{theorem}
\begin{proof}
Let $n \geq 2$ be fixed for the remainder of this proof. 
Gura, Hirst, and Mummert~\cite{gura}*{Theorem 6.6} proved that $
\mathsf{GAC}_n \equivsW \mathsf{C}_\nat
$.
Therefore, it is sufficient to establish the reductions 
$
\mathsf{GAC}_n \leqsW \mathsf{EMB}_n$ and $\mathsf{EMB}_n \leqsW \mathsf{C}_\nat$. 

The reduction $\mathsf{GAC}_n \leqsW \mathsf{EMB}_n$ follows from the proof of 
Theorem~\ref{w1}, because the construction there produces 
an e-matroid whose dimension is the same as the number of components of the graph. 

To show that $\mathsf{EMB}_n \leqsW \mathsf{C}_\nat$, let $M$ be an 
e-matroid on the set $\nat$ with some basis of size~$n$.  
Consider a sequence $(F_i)_{i \in \nat}$ of finite subsets of $M$ so
that $F_{t}$ consists of the first element
of $\nat^n$ (under the lexicographical order on increasing sequences read right to left)
that is not one of the sets
$e(i)$ for any~$i < t$. Thus $F(0) = \{0, \ldots, n-1\}$, and $F_{t+1}$ will
differ from $F(t)$ exactly when $e(t) = F_t$. Note that, because there is an independent
set of size $n$, there will be a $t$ such that $F_s = F_t$ for all $s > t$. 
Let $S$ be the set of all $t > 0$ for which $F_s = F_t$ for all $s > t$ and let
$f$ be a function whose range is the complement of $S$. We may apply
$\mathsf{C}_\nat$ to $f$ to find a $t \in S$; then $F_t$ is a basis for~$M$. 
\end{proof}

The next lemma, which is well known, extends the list of principles in Theorem~\ref{w2} slightly,
simplifying the proof of the next theorem.

\begin{lem}\label{Q14A}
Let $\mathsf{C}^u_\nat$ denote the restriction of $\mathsf{C}_\nat$ to functions for which
the complement of the range consists of a unique natural number.  Then
$\mathsf{C}^u_\nat \equivsW \mathsf{C}_\nat$.
\end{lem}

\begin{proof}
Because $\mathsf{C}^u_\nat$ restricts $\mathsf{C}_\nat$ to a smaller class of inputs,
$\mathsf{C}^u_\nat \leqsW \mathsf{C}_\nat$.  To prove $\mathsf{C}_\nat \leqsW \mathsf{C}^u_\nat$,
suppose $f \colon \nat \to \nat$ is not surjective.  In the following construction, we will conflate
the pair $(i,j)$ with its integer code via a fixed bijection between $\nat$ and $\nat \times \nat$.
Define $g\colon \nat \to \nat$ by the following moving marker construction.  Let $m_0 = (0,0)$ be the initial
marker.  Suppose $m_k = (m^0_k , m^1_k )$ has been defined.  If $f(k) \neq m^0_k$, set
$m_{k+1} = m_k$ and set $g(k)$ to the least code for a pair not included in
$\{g(j) : j<k \}$.  If $f(k)=m^0_k$, define a pair $(y_0, t_0)$ so that  
\[
\begin{split}
y_0 & = (\mu \, y \le k+1)(\forall j \le k)[f(j) \neq y)],\\
t_0 &= (\mu \, t)( \forall j<k)[g(j) \neq (y_0, t)],
\end{split}
\]
and then set $m_{k+1} = (y_0 , t_0 )$ and $g(k) = m_k$.

Intuitively, if $y$ is the smallest natural number not in the range of $f$, then at some stage in the construction
the marker is set to $(y,n)$ for some~$n$, and does not move after that point.  The code $(y,n)$ is not in the
range of~$g$, but every other code and consequently every other natural number is in the range of~$g$.
Thus $g$ satisfies the input requirements for $\mathsf{C}^u_\nat$, and the process yields $(y,n)$ as an output.
The number $y$ (retrievable by a projection function) is a solution to $\mathsf{C}_\nat$ for input~$f$.
\end{proof}

The following theorem adds the fixed dimension vector space basis problem to the list of equivalent problems
of Theorem~\ref{w2}

\begin{theorem}\label{Q14B}
For $n\ge 2$, $\mathsf{VSB}_n \equivsW \mathsf{C}_\nat$.
\end{theorem}

\begin{proof}
By Theorem~\ref{w2}, $\mathsf{EMB_n} \leqsW \mathsf{C}_\nat$.  In the proof of Theorem~\ref{w1},
the argument showing $\mathsf{VSB} \leqsW \mathsf{EMB}$ preserves the dimension of input
vector space, and so shows $\mathsf{VSB}_n \leqsW \mathsf{EMB}_n$.  By transitivity,
$\mathsf{VSB}_n \leqsW\mathsf{C}_\nat$.

Next we will show that $\mathsf{C}^u_\nat \leqsW \mathsf{VSB}_2$.  Our proof uses
ideas and notation from the proof of Theorem III.4.2 of Simpson~\cite{simpson}.
Fix $f\colon \nat\to \nat$ with the range of $f$ including all of $\nat$ except for one value.
Let $V_0$ be the set of all formal sums $\sum_{i \in I} q_i x_i$ with $I$ finite
and $0\neq q_i \in \mathbb Q$.  We can identify formal sums with their
sequence codes, yielding a well-ordering on $V_0$.  Without loss of generality,
we may assume that $x_i$ is minimal in this ordering among all vectors
with a nonzero coefficient on $x_i$.  As in Simpson's proof, let
$x_m^\prime = x_{2f(m)}+ (m+1)x_{2f(m)+1}$ and $X^\prime = \{x_m^\prime :  m \in \nat \}$.
Let $U_0$ denote the subspace consisting of the linear span of $X^\prime$.
Note that $\sum_{i \in I} q_i x_i \in U_0$ if and only if
\[
(\forall n) \left[ (q_{2n} \neq 0 \to f(q_{2n+1}/q_{2n}-1) = n)
\land
(q_{2n} = 0 \to q_{2n+1} = 0) \right],
\]
so $U_0$ is computable from $f$.  Let $V_1$ be $V_0 / U_0$, where a vector $v$ is in $V_1$ if and only if
it is the element of $\{v-u : u \in U_0 \}$ which is least in the well ordering on $V_0$.  Only finitely
many sequence codes are less than the code for $v$, so $V_1$ is computable.

By our choice of ordering and the construction of $U_0$, for every $i\in \nat$, $x_{2i} \in V_1$.
Let $U_1$ be the linear span of $\{ x_{2i} : i \in \nat \}$ in $V_1$.  Then $U_1$ is a vector
subspace of $V_1$ computable from $f$, and we may construct the quotient space
$V= V_1 / U_1$, using minimal representatives as before.  For any $j \in \nat$,
\[
x_0 = x_{2f(j)+1} - (-\frac{1}{j+1} x_{2f(j)} - x_0 ) - \frac{1}{j+1}(x_{2f(j)} + (j+1) x_{2f(j) + 1}).
\]
The vector $-\frac{1}{j+1} x_{2f(j)} - x_0$ is in $U_1$ and $\frac{1}{j+1}(x_{2f(j)} + (j+1) x_{2f(j) + 1})$ is in $U_0$,
so $x_0$ and $x_{2f(j)+1}$ correspond to the same vector in $V$.  The range of $f$ excludes only
one element, so the dimension of $V$ is $2$.  Let $\{v_1 , v_2\}$ be a basis for $V$.
Let $P$ be the finite collection of odd indices in the formal sums for $v_1$ and $v_2$, and let
$R = \{m : 2m+1 \in P \}$.  Exactly one $m$ in $R$ does not appear in the range of $f$.
Thus, for exactly one $m$ in $R$, $\{ x_0 , x_{2m+1}\}$ is linearly independent.
Sequentially enumerate linear combinations of the form $q_0 x_0 + q_1 x_{2m+1}$, ejecting values from $R$
corresponding to linear combinations that equal $0$ in $V$.  The last value left in $R$ is
the sole natural number
that is not in the range of $f$.  Thus $\mathsf C^u_\nat \leqsW \mathsf {VSB}_2$.
By Lemma~\ref{Q14A}, $\mathsf{C}_\nat \leqsW \mathsf {VSB}_2$.

To prove $\mathsf{C}_\nat \leqsW \mathsf{VSB}_n$ for $n>2$, add $n-1$ dummy vectors to the the basis of $V_0$
in the preceding argument.
\end{proof}

The reduction of $\mathsf{EMB}_n$ to $\mathsf{C}_\nat$ in the proof of Theorem~\ref{w2} relies heavily on knowing the precise dimensions (in the appropriate sense) of the objects being studied. This suggests a variation in which we only place an upper bound on their  dimensions.  We begin with definitions of bounded versions of some Weihrauch principles.

\newcommand{\Cmax}{\mathsf{C}^{\subset}_{\text{max}}}
\newcommand{\Ccmax}{\mathsf{C}^{\#}_{\text{max}}}

\begin{defn}
We define the following Weihrauch principles. 
\begin{itemize}
\item $\mathsf{EMB}_{<\omega}$: the bounded e-matroid basis problem.
\begin{align*}
\mathsf{EMB}_{<\omega } = \{ (n,M,B) :~ &n\in\nat,~M \text{ is an e-matroid,}
\operatorname{rank}(M)\le n,\\
& \text{and }  B \text{ is a basis for }M\}
\end{align*}

\item
$\mathsf{GAC}_{<\omega }$:  The bounded graph antichain problem.  Letting 
$\operatorname{Max}(G)$ be the set of maximal antichains of $G$, we have
\begin{align*}
\mathsf{GAC}_{<\omega }= \{(n,G,A)~:~& n\in \nat,~G \text{ is a graph,}\\
&\text{each set of } n \text { vertices has a path connected pair,}\\
&\text{and } A \in \operatorname{Max}(G)\}
\end{align*}
\item
$\Cmax$:  Picking a maximal element (relative to the containment partial ordering)
in the complement of an enumeration of finite nonempty sets whose range includes all sets larger than some bound.
\begin{align*}
\Cmax =\{(n,f,X)~:~&n \in \nat,~f\colon\nat\to [\nat]^{<\nat}_{\neq\emptyset},~X\in [\nat]^{<\nat},\\  %[]jlh4/17
&\operatorname{range}(f) \text{~includes all sets of cardinality}\ge n,\\
& X \notin\operatorname{range}(f), \text{ and }\\
& (\forall Y \in \nat^{<\nat})[Y \supsetneq X \to Y\in \operatorname{range}(f)]\}
\end{align*}
\item
$\Ccmax$:  Picking an element of maximal cardinality in the complement of an enumeration of finite nonempty sets
whose range includes all sets larger than some bound.
\begin{align*}
\Ccmax=\{(n,f,X)~:~&n \in \nat,~f\colon\nat\to [\nat]^{<\nat}_{\neq\emptyset},~X\in [\nat]^{<\nat},\\  %[]jlh4/17
&\operatorname{range}(f) \text{~includes all sets of cardinality}\ge n,\\
& X \notin\operatorname{range}(f), \text{ and }\\
& (\forall Y \in \nat^{<\nat})[|Y|>|X| \to Y\in \operatorname{range}(f)]\}
\end{align*}

\end{itemize}
\end{defn}

\begin{theorem}\label{w3} The following strong Weihrauch equivalences hold:
\[
\mathsf{EMB}_{<\omega }
\equivsW
\mathsf{GAC}_{<\omega }
\equivsW
\Cmax
\equivsW
\Ccmax
\]
\end{theorem}

\begin{proof}
We will prove each of the following reductions, proceeding from right to left.
\[
\Cmax \leqsW
\Ccmax \leqsW
\mathsf{GAC}_{<\omega }\leqsW
\mathsf{EMB}_{<\omega }\leqsW
\Cmax
\]
To prove \(\mathsf{EMB}_{<\omega }\leqsW \Cmax\),
suppose $(M,e)$ is an e-matroid such that every subset of $M$ of size at least $n$ is in the range
of $e$.  Let $\{X_j : j \in \nat\}$ be an enumeration of $[\nat]^{<\nat}$ and let $(i,j)$ denote the
output of a bijective pairing function.  Note that every $m\in \nat$ has a unique representation
of the form $2(i,j)+\varepsilon$ where $i,j \in \nat$ and $\varepsilon \in \{0,1\}$.  
Define $f\colon \nat \to [\nat]^{<\nat}$ by
\[
f(2(i,j)+ \varepsilon)=
\begin{cases}
X_j \text{ if }\varepsilon= 0 \land i\notin M \land i \in X_j, \text{ and}\\
e((i,j)) \text{ otherwise.}
\end{cases}
\]
The range of $f$ consists of the range of $e$ plus all finite sets containing any elements of the complement of~$M$.
Apply $\Cmax$ to $f$ to obtain a finite set $B\subseteq \nat$ in the complement of the range of $f$
that is maximal with respect to the containment partial ordering.  The range of $f$ includes all finite sets containing
elements of the complement of $M$, so $B \subseteq M$.  Furthermore, the range of $f$ includes the range of $e$,
so $B$ is independent in $(M,E)$.  By maximality of $B$, $B$ spans $(M,e)$, so $B$ is a basis for $(M,e)$.

To prove $\mathsf{GAC}_{<\omega }\leqsW
\mathsf{EMB}_{<\omega }$, emulate the reduction of $\mathsf{GAC}$ to $\mathsf{EMB}$ from the proof of
Theorem~\ref{w1}.  Because $G$ has at most $n$ connected components,
every set of $n+1$ elements in the related matroid is dependent and so appears in the range of the enumeration.

To prove $\Ccmax\leqsW
\mathsf{GAC}_{<\omega }$, suppose $f\colon \nat \to [\nat]^{<\nat}_{\not = \emptyset}$ and the range of $f$
includes all finite subsets of cardinality at least $n$.  For each $b$ with $1\leq b<n$, let $g_b \colon \nat \to \nat^{<\nat}$ be an enumeration
of all subsets of $\nat$ of cardinality exactly~$b$.  We will construct a graph $G$ consisting of $n-1$ subgraphs
each with one or two connected components.  The vertices of $G$ are
$\{ u^b _j , v^b_j : 1 \leq b < n \land j \in \nat \}$.  For each $b$ with $1 \leq b<n$ and each $j \in \nat$,
add the edge $(u^b_j , u^{b}_{j +1})$ to the edge set $E$ of $G$.  For each $b$ with $1\leq b < n$,
define $k^b_0 = 0$.  Suppose $k^b_j$ is defined.
If $(\exists t \leq j)[ f(t) = g_b (k^b_j )]$, add $(v^b_j , u^b_{j})$ to $E$ and set $k^b_{j+1} = k^b_j + 1$.
Otherwise, if $(\forall t\leq j)[f(t) \neq g_b(k^b_j)]$, add $(v^b_j , v^b_{j+1})$ to $E$ and set $k^b_{j+1} = k^b_j$.
Note that the graph $G$ is uniformly computable from $f$.

Apply $\mathsf{GAC}_{<\omega }$ to find a maximal (finite) antichain $D$ in $G$.
Let $b_0$ be the largest number less than $n$ such that $D$ contains two vertices with superscript~$b_0$.
(If no such $b_0$ exists, $\emptyset$ is the largest set in the complement of the range of $f$.)
At least one of these vertices must be $v_j^{b_0}$ for some $j$.  Let $j_0$ be the largest value such that
$v^{b_0}_{j_0} \in D$.  Then $g_{b_0} ( k^{b_0}_{j_0} )$ is a set of maximal cardinality in the complement of the range of $f$.

To conclude the proof, we need only show that $\Cmax \leqsW
\Ccmax$.  Any $f$ and $n$ satisfying the hypotheses of $\Cmax$ also satisfy those
of $\Ccmax$.  Any subset in the complement of the range of $f$ that is maximal in cardinality is also maximal
with respect to the containment partial ordering, so the identity functionals witness the desired reduction.
\end{proof}

We close our discussion of Weihrauch reducibility with the following theorem that adds $\mathsf{VSB}_{<\omega}$ to 
the equivalences of Theorem~\ref{w3}.  Here $\mathsf{VSB}_{<\omega}$ is the problem which, given an input of $n\in \nat$
and a vector space in which every set of $n$ vectors is linearly dependent, returns a basis for the vector space.

\begin{theorem}\label{Q17A}
$\mathsf{VSB}_{<\omega} \equivsW \Cmax$.
\end{theorem}

\begin{proof}
By Theorem~\ref{w3}, $\mathsf{EMB}_{<\omega} \leqsW \Cmax$.  The
proof of $\mathsf{VSB} \leqsW \mathsf{EMB}$ in Theorem~\ref{w1} preserves dimension,
so that argument also witnesses that $\mathsf{VSB}_{<\omega} \leqsW \mathsf{EMB}_{<\omega}$.
By transitivity, $\mathsf{VSB}_{<\omega} \leqsW \Cmax$.

Next we will adapt arguments from the proofs of Lemma~\ref{Q14A} and Theorem~\ref{Q14B}
to show that $\Ccmax \leqsW \mathsf{VSB}_{<\omega}$.  Fix $n$ and
$f\colon  \nat \to [\nat] ^{<\nat}$ such that the range of $f$ includes all sets of cardinality $\ge n$.%[]jlh4/17
For each $j<n$, let $h_j$ be a bijective enumeration of $\{X : X \subset \nat \land j \le |X| <n\} \times \nat$.
Emulating the moving marker construction of Lemma~\ref{Q14A}, for each $j<n$ define $g_j$ such that
either the range of $f$ includes all sets of cardinality $k$ for $j \le k <n$ and $g_j$ is surjective
or the unique value not in the range of $g_j$ is some $m$ such that $h_j (m) = (X_0, m_0 )$
where $j \le |X_0 |< n$ and $X_0$ is in the complement of the range of $f$.
(For use in the proof of Theorem~\ref{RMVW}, note that the convergence of the moving marker
construction can be formally proved using the collection principle $\mathsf{B}\Sigma^0_1$, which
is provable in $\rca$.)

Now we carry out an $n$-fold analog of the vector space construction in the proof of Theorem~\ref{Q14B}. The goal of the construction is to form a space $V$ as a direct sum of subspaces 
$W_i$, $i < n$, such that if $j_0$ is the largest size of a set omitted from the range of~$f$, then the dimension of $W_i$ is $1$ for $i > j_0$ and the dimension is $2$ for $i \leq j_0$.  This will ensure that the dimension of $V$ is finite, and moreover will allow us to compute the value of~$j_0$ if we know the exact dimension of~$V$.

Let $V_0$ be the set of formal sums
$\sum_{(i,k)\in I_k\times[0,n)} q_{(i,k)} x_{(i,k)}$ where for each $k<n$, each $I_k$ is finite
and $0 \neq q_{(i,k)} \in \mathbb Q$.  Identifying $h_j(m)=(X_0,m_0)$ with the integer code for the pair,
for each $k<n$ and each $m$, let
$x^\prime_{(m,k)} = x_{(2h_k (m),k)} + (m+1)x_{(2h_k (m) + 1, k)}$ and
$X^\prime = \{ x^\prime_{(m,k)} : m \in \nat \land k<n\}$.  Let
$U_0$ be the linear span of $X^\prime$ and set $V_1 = V_0 / U_0$.
Let $U_1$ be the linear span in $V_1$ of $\{ x_{(2m,k)} : m \in \nat \land k<n\}$
and let $V = V_1 / U_1$.  Then $V$ has a two dimensional subspace corresponding to
each $j<n$ such that the range of $f$ omits a set of cardinality $k$ with $j \le k <n$,
and a one dimensional subspace corresponding to each $j<n$ such that $f$ maps $\nat$ onto
the sets of cardinality $k$ with $j \le k < n$.  Thus the dimension of $V$ is between $n$ and $2n$, and
any set of $2n+1$ vectors is linearly dependent.

(For use in the proof of Theorem~\ref{RMVW}, note that the claim that any collection of $2n+1$ vectors
of $V$ is linearly dependent can be proved in $\rca$ as follows.  Fix a set of $2n+1$ nonzero vectors,
$S = \{ u_0 , \dots , u_{2n} \}$.  Let $B_0$ be the finite set of those vectors of the form
$x_{(i,k)}$ that appear in the sums representing each $u_i$.  Because $S$ is finite,
$\Sigma^0_1$ induction suffices to find the smallest subset of $B_0$ that spans $S$.
Call this set $B_1$.  By minimality, $B_1$ is linearly independent.  For each $k<n$, the function
$g_k$ omits at most one value, so $B_1$ contains at most two vectors of the form $x_{(i,k)}$.
Thus $|B_1|\le 2n$.  Let $B_1 = \{ v_0 , \dots , v_j \}$ where $j<2n$.  The vectors of $B_1$ span $S$,
so $u_0 = \sum_{i\le j} c_i v_i$, with some $c_{i_0} \neq 0$.  Solving for $v_{i_0}$, we see that
$v_{i_0}$ is in the span of $B_2 = \{ u_0 \} \cup B_1\setminus \{v_{i_0}\}$.  Thus $B_2$ is a linearly
independent set spanning $S$.  Iterating this process by primitive recursion, we eventually find a $u_m \in S$
which is a linear combination of $\{ u_i : i<m\}$.  Thus $S$ is linearly dependent.)

Apply $\mathsf{VSB}_{< \omega}$ to find a basis $B$ for $V$.  Then $k=|B|-n-1$ is the cardinality of the largest
set omitted from the range of $f$.  Let $P$ be the finite collection of odd numbers $m$ such that
$(m,k)$ appears as an index in a formal sum for an element of $B$.  Let $R=\{ 2m+1\in P\}$.
Exactly one $m$ in $R$ does not appear in the range of $g_k$.
Thus for exactly one $m$ in $R$, $\{x_{(0,k)} , x_{(2m+1 , k)} \}$ is linearly independent.
Sequentially examine linear combinations of the form $q_0 x_{(0,k)} + q_1 x_{(2m+1 , k)}$, ejecting
values from $R$ corresponding to vectors equal to $0$ in $V$, until only one is left.
Viewed as a code for a pair, the first component of that value is a code for
a set of maximum cardinality in the complement of the range of $f$.  Thus
$\Ccmax \leqsW \mathsf{VSB}_{<\omega}$.  By Theorem~\ref{w3},
$\Cmax \leqsW \mathsf{VSB}_{<\omega}$.
\end{proof}

\section{Reducibility and Reverse Mathematics}

We conclude by extracting a final reverse mathematics result from the proofs of Theorem~\ref{w3}
and Theorem~\ref{Q17A}, extending the list of equivalences in Theorem~\ref{b2}.

\begin{theorem}[$\rca$]\label{RMVW}
$(\rca)$ The following are equivalent:
\begin{enumlist}
\item \label{RMVW1}  $\mathsf{I}\Sigma^0_2$, the induction scheme for $\Sigma^0_2$ formulas with set parameters.
\item \label{RMVW2}  Let $V$ be a countable vector space such that for some $n$, every
subset of $n$ vectors is linearly dependent.  Then $V$ has a basis.

\item \label{RMVW3} A formalized version of~$\Ccmax$. Suppose $f\colon\nat \to [\nat]^{<\nat}_{\neq \emptyset}$ and
there is an $n$ such that for all $X \in \nat^{<\nat}$,
$[|X| \ge n \to \exists t (f(t) = X) ]$.
Then there is an $X \in [\nat]^{<\nat}$ such that
$(\forall t)[f(t) \neq X$ and for all $Y \in [\nat]^{<\nat}$, $[|X|<|Y| \to \exists t( f(t) = Y)] $.

\item \label{RMVW4} A formalized version of~$\Cmax$. Suppose $f\colon\nat \to [\nat]^{<\nat}_{\neq \emptyset}$ and
there is an $n$ such that for all $X \in \nat^{<\nat}$,
$ (|X| \ge n \to \exists t (f(t) = X) )$.
Then there is an $X \in [\nat]^{<\nat}$ such that
$(\forall t )[f(t) \neq X] $ and for all $Y \in \nat^{<\nat}$, $[X \subsetneq Y \to \exists t( f(t) = Y)]$.
\end{enumlist}
\end{theorem}

\begin{proof}
First, we use~\ref{RMVW1} to prove~\ref{RMVW2}.
If $V$ is a vector space and every set of $n$ vectors is linearly dependent, the construction from the proof of Theorem~\ref{w1}
can be formalized to yield an e-matroid of rank no more than~$n$.  By Theorem~\ref{b2}, $\mathsf{I}\Sigma^0_2$ implies
that this matroid has a basis which is also a basis of $V$.

To show that~\ref{RMVW2} implies~\ref{RMVW3}, formalize the argument form the proof of Theorem~\ref{Q17A}
showing that $\Ccmax \leqsW \mathsf{VSB}_{<\omega}$, using the parenthetical comments.
As noted, the convergence of the moving marker construction is provable in $\rca$, as is the claim that
every set of $2n+1$ vectors is linearly dependent.

The proof that~\ref{RMVW3} implies~\ref{RMVW4} follows immediately from the fact that any set that is maximal
in the sense of~\ref{RMVW3} is automatically maximal in the sense of~\ref{RMVW4}.

The proof that $\mathsf{EMB}_{<\omega} \leqsW \Cmax$ from Theorem~\ref{w3}
can be formalized in $\rca$ to show that~\ref{RMVW4} implies item~\ref{b2a} of Theorem~\ref{b2}.  By
Theorem~\ref{b2}, this implies $\mathsf{I}\Sigma^0_2$, completing the proof.
\end{proof}

\bibliographystyle{amsplain}

\end{document}